\documentclass[preprint]{elsarticle}

\usepackage{lineno,hyperref}
\usepackage{amsmath,amsthm,amssymb,amsfonts}
\usepackage[all]{xy}
\newdefinition{defn}{Definition}[section]
\newdefinition{nim}[defn]{}
\newdefinition{rem}[defn]{Remark}
\newdefinition{ex}[defn]{Example}
\newtheorem{thm}{Theorem}[section]

\newtheorem{prop}[thm]{Proposition}
\newcommand{\HH}{\mathrm{HH}}
\newcommand{\Hc}{\mathrm{H}}
\newcommand{\res}{\mathrm{res}}
\newcommand{\Hom}{\mathrm{Hom}}
\newcommand{\Ob}{\mathrm{Ob}}
\newcommand{\Inj}{\mathrm{Inj}}
\newcommand{\Mor}{\mathrm{Mor}}
\newcommand{\Ima}{\mathrm{Im}}
\newcommand{\Aut}{\mathrm{Aut}}
\newcommand{\Ker}{\mathrm{Ker}}
\newcommand{\Br}{\mathrm{Br}}
\newcommand{\Id}{\mathrm{Id}}
\newcommand{\Ext}{\mathrm{Ext}}
\modulolinenumbers[5]

\begin{document}

\begin{frontmatter}

\title{A theorem of Mislin for cohomology of fusion systems and applications to block algebras of finite groups}

\author{Constantin-Cosmin Todea}\corref{mycorrespondingauthor}
\address{Department of Mathematics,
         Technical University of Cluj-Napoca,
         Str. G. Baritiu 25,
         Cluj-Napoca 400027,
         Romania}



\cortext[mycorrespondingauthor]{Corresponding author}
\ead{constantin.todea@math.utcluj.ro}


\begin{abstract}
The aim of this short expository article is to give an algebraic proof for a theorem of Mislin in the case of cohomology of saturated fusion systems defined on $p$-groups when $p$ is odd. Some applications of this theorem, which includes different proofs of known results regarding block algebras of finite groups, are also given.
\end{abstract}


\end{frontmatter}


\section{Introduction}\label{sec1}

A saturated fusion system $\mathcal{F}$ on a finite $p$-group $P$ is a category whose objects are the subgroups of $P$ and whose morphisms satisfy certain axioms mimicking the behavior of a finite group $G$ having $P$ as a Sylow subgroup. For the convenience of the reader we give, in Section \ref{sec1'}, the definition of saturated fusion systems and some basic facts, following \cite{AsKeOl}.  Let $k$ be an algebraically closed field of characteristic $p$. We denote by $\Hc^*(G,k)$ the cohomology algebra of a group $G$ with trivial coefficients. We denote by $\Hc^*(\mathcal{F})$ the subalgebra of $\mathcal{F}$-stable elements in $\Hc^*(P,k)$, i.e. the cohomology algebra of the saturated fusion system $\mathcal{F}$, which is the subalgebra of $\Hc^*(P,k)$ consisting of elements $\zeta\in \Hc^*(P,k)$ such that
$$\res^P_Q(\zeta)=\res_{\varphi}(\zeta),$$
for any $\varphi\in\Hom_{\mathcal{F}}(Q,P)$ and any subgroup $Q$ of $P$.

A celebrated theorem of Mislin in \cite{Mis} regarding the control of fusion in group cohomology (stated for compact Lie groups) has now a new short algebraic proof for $p$ odd thanks to Benson, Grodal and Henke \cite{BenGroHen}. See \cite{Hid}, \cite{Oku} for other algebraic proofs which uses Mackey functors and cohomology of trivial source modules; see also  \cite{Sym} for a different algebraic approach. Also, in \cite[Remark 5.8]{LiTr} Linckelmann suggests a topological proof for Mislin's theorem in the case of block algebras of finite groups, more precisely for cohomology of fusion systems associated to blocks. We prove this theorem of Mislin in the general context of cohomology of saturated fusion systems for $p$ odd, by extending the proof of Benson, Grodal and Henke to saturated fusion systems.

The $k$-algebra $\Hc^*(\mathcal{F})$ is a graded-commutative and finitely generated, hence we associate the spectrum of maximal ideals, i.e. the algebraic variety which we denote by $V_{\mathcal{F}}$. Let $\mathcal{G}$ be a saturated fusion subsystem of $\mathcal{F}$ defined on the same finite $p$-group $P$. We have an inclusion map
$$i:\Hc^*(\mathcal{F})\rightarrow\Hc^*(\mathcal{G}),$$ which induces a map on varieties
$$i^*:V_{\mathcal{G}}\rightarrow V_{\mathcal{F}}.$$
The main result of this paper is the following theorem which contains Mislin's theorem for saturated fusion systems as a special case, when $p$ is odd.
\begin{thm}\label{thmMislin}Let $\mathcal{G}$ be a saturated fusion subsystem of $\mathcal{F}$ defined on the same finite $p$-group $P$ and $p$ an odd prime. If for each $\zeta\in \Hc^*(\mathcal{G})$ we have $\zeta^{p^r}\in\Ima(i)$ for some $r\geq 0$, then $\mathcal{G}=\mathcal{F}$. In particular we have  $\Hc^*(\mathcal{F})=\Hc^*(\mathcal{G})$ if and only if $\mathcal{G}=\mathcal{F}$.
\end{thm}

The ingredients for the proof of the above theorem were already mentioned by Benson, Grodal and Henke in \cite[Remark 3.7]{BenGroHen}. We follow their suggestion and we fill this gap in the literature. There are two ingredients for proving Theorem \ref{thmMislin}. To explain those ingredients, we first introduce the following terminology. Let $\mathcal{G}$ be a saturated fusion subsystem of $\mathcal{F}$ defined on the same finite $p$-group $P$. For shortness, we will say that $\mathcal{G}$ \emph{controls $p$-fusion in $\mathcal{F}$ on elementary abelian $p$-subgroups} if $\Hom_{\mathcal{G}}(E_1,E_2)=\Hom_{\mathcal{F}}(E_1,E_2)$ for all $E_1,E_2\leq P$, where $E_1,E_2$ runs over the set of elementary abelian $p$-subgroups of $P$. Now one of the ingredients is a property of saturated
fusion systems (\cite[Theorem B]{BenGroHen}) which says that if $p$ is odd and $\mathcal{G}$ controls $p$-fusion in $\mathcal{F}$ on elementary abelian $p$-subgroups then $\mathcal{G}=\mathcal{F}$. Another ingredient is the following theorem which says that control of $p$-fusion on elementary abelian subgroups happens if and only if $i^*$ is a bijective map. \cite[Theorem 2]{AlpMis} and \cite[Proposition 10.9]{QUill2} are similar statements for group cohomology.

\begin{thm}\label{thmbij}Let $\mathcal{G}$ be a saturated fusion subsystem of $\mathcal{F}$ defined on the same finite $p$-group $P$. Then $i^*$ is surjective. Moreover we have that $i^*$ is an injective map if and only if $\mathcal{G}$ controls $p$-fusion in $\mathcal{F}$ on elementary abelian $p$-subgroups.
\end{thm}
To prove Theorem \ref{thmbij} we will use Quillen stratification for cohomology of saturated fusion systems given by Markus Linckelmann in \cite{LiQuillfus}, for which the proof is the same as for block algebras \cite{LiQuill}, with some minor adjustments. Since it has not appeared in this form in the literature we state it here for completeness. For any subgroup $Q$ of $P$ denote by $V_Q$ the maximal ideal spectrum of $\Hc^*(Q,k)$, and set $V^+_Q = V_Q \setminus\bigcup_{R<Q} (\res^Q_R)^*(V_R)$. Denote by $V_{\mathcal{F},Q}$ and $V^+_{\mathcal{F},Q} $ the images of $V_Q$ and $V^+_Q$ in $V_{\mathcal{F}}$ under the map $r^*_{\mathcal{F},Q}$ induced by the algebra homomorphism $r_{\mathcal{F},Q}:\Hc^*(\mathcal{F})\rightarrow \Hc^*(Q, k)$ given by composing the inclusion $\Hc^*(\mathcal{F})\subseteq \Hc^*(P,k)$
with the restriction $\res^P_Q:\Hc^*(P,k)\rightarrow \Hc^*(Q,k).$
\begin{thm}\label{thmA1}(\cite[Theorem 1]{LiQuillfus}) With the notation above, the following hold.
\begin{itemize}\item[(i)] The variety $V_{\mathcal{F}}$ is the disjoint union of the locally closed subvarieties $V^+_{\mathcal{F},E}$, where $E$ runs over a set of representatives of the $\mathcal{F}$-isomorphism classes of elementary abelian subgroups of $P$.
\item[(ii)] Let $E$ be an elementary abelian subgroup of $P$. The group $W_{\mathcal{F}}(E)=\Aut_{\mathcal{F}}(E)$ acts on $V^+_E$ and the restriction map $\res^P_E$ induces an inseparable isogeny $$V^+_E /W_{\mathcal{F}}(E)\rightarrow V^+_{\mathcal{F},E}.$$
\end{itemize}
\end{thm}
We will not give the proof of Theorem \ref{thmA1} from \cite{LiQuillfus}, since the way to obtain it is to copy the proof from \cite{LiQuill} in the context of saturated fusion systems. This proof follows in turn very closely Benson's presentation in  \cite{BenII} of parts of
Quillen's original work in \cite{QUill2}, with only additional ingredient the fusion stable bisets whose
existence was proved by Broto, Levi, and Oliver in \cite{BLO}.
\begin{rem}\label{remfree} From \cite{Par} we know that any saturated fusion system $\mathcal{F}$ can be identified with $\mathcal{F}_P(G)$ for some finite group $G$ containing $P$ as a subgroup, where in the last category the morphisms are given by conjugation with elements from $G$. It follows that for any elementary abelian $p$-subgroup $E$ of $P$ we can identify $\Aut_{\mathcal{F}}(E)$ with $N_G(E)/C_G(E)$. Now, arguments similar to \cite[Theorem 9.1.7]{Eve} assure us that $\Aut_{\mathcal{F}}(E)$ acts freely on $V^+_E$.
\end{rem}
In Section \ref{sec2} we will prove Theorem \ref{thmMislin} and \ref{thmbij}. In Section \ref{sec3} we give applications of Mislin's theorem to the case of block algebras of finite groups. It is known that if two blocks are basic Morita equivalent then there is an isomorphism between the defect groups which induces an equivalence between their local categories \cite[7.6.6]{Pu}.  We will prove that in some cases Mislin's theorem  and an inclusion between some subalgebras of stable elements in Hochschild cohomology algebras of group algebras over the defect groups (with respect to the source algebras of the blocks; see \cite{LiTr}) can give us an alternative method  to obtain an equality between the associated fusion systems (Brauer categories). Mislin's Theorem and cohomological techniques have implications for other results involving block algebras of finite groups. For example, under some assumptions (which are required to apply Mislin's theorem as above) we will show that two block algebras, which are splendid stable equivalent in the sense of Linckelmann \cite{LiSpl}, have the same fusion systems (Theorem \ref{thmsplendidstable}). One of the main features of this article appears in Section \ref{sec3} where we show how homological methods have applications in proving old/new results for block algebras of finite groups, although some of these results are already known and proved by specific methods.
\section{Saturated fusion systems}\label{sec1'}
The axioms of saturated fusion systems were invented by Puig in early 1990's. He called them "Frobenius categories" but his approach was taken up and extended by others and in the process an alternate terminology evolved which is now commonly used. In particular, we call today a saturated fusion system what Puig refers to a  Frobenius category. The next definitions which we recall are modified versions of \cite{Pu6}, but equivalent to them.

\begin{defn}\label{defnfus}(\cite{Pu6}, \cite{BLO}) A \emph{fusion system} over a $p$-group $S$ is a category $\mathcal{F}$, where $\Ob(\mathcal{F})$ is the set of all subgroups of $S$ and which satisfies the following two properties for all $P,Q\leq S$:
\begin{itemize}
  \item[$\bullet$]$\Hom_S(P,Q)\leq \Mor_{\mathcal{F}}(P,Q)\leq \Inj(P,Q)$ and
  \item[$\bullet$] each $\varphi\in\Mor_{\mathcal{F}}(P,Q)$ is the composite of an $\mathcal{F}$-isomorphism followed by an inclusion.
\end{itemize}
\end{defn}
Here $\Inj(P,Q)$ is the set of all injective group homomorphisms from $Q$ to $P$ and $\Hom_S(P,Q)$ is the set of group homomorphisms
$$\{\varphi\in\Hom(P,Q)| \varphi=c_s~\text{for some}~s\in S~\text{such that}~{}^gP\leq Q\}$$
where $c_s:P\rightarrow Q$ is defined by $c_s(x)={}^xs=xsx^{-1}$ for any $x\in P$. Composition in a fusion system $\mathcal{F}$ is always given by the composition of homomorphisms. We usually write $\Hom_{\mathcal{F}}(P,Q)=\Mor_{\mathcal{F}}(P,Q)$ to emphasize that morphisms in  $\mathcal{F}$ are actually group homomorphisms and we set $$\Aut_{\mathcal{F}}(P)=\Hom_{\mathcal{F}}(P,P).$$
A saturated fusion system is a fusion system satisfying certain axioms  which had origins in some properties of fusion in finite groups. The following version of these axioms is due to Roberts and Shpectorov and is taken from \cite{AsKeOl}.
\begin{defn} (\cite[Definition 2.2]{AsKeOl}) Let $\mathcal{F}$ be a fusion system over a $p$-group $S$.
\begin{itemize}
  \item[$\bullet$] Two subgroups $P,Q\leq S$ are $\mathcal{F}$-\emph{conjugate} if they are isomorphic as objects of $\mathcal{F}$.
      \item[$\bullet$] A subgroup $P\leq S$ is \emph{fully automized} in $\mathcal{F}$ if $\Aut_S(P)$ is a Sylow $p$-subgroup of $\Aut_{\mathcal{F}}(P)$.
  \item[$\bullet$] A subgroup $P\leq S$ is \emph{receptive} in $\mathcal{F}$ if it has the following property: for each $Q\leq S$ and each $\varphi\in\Mor_{\mathcal{F}}(Q,P)$ an isomorphism if we set
       $$N_{\varphi}=\{g\in N_S(Q)|{}^{\varphi}c_g\in\Aut_S(P)\},$$
       then there is $\overline{\varphi}\in\Hom_{\mathcal{F}}(N_{\varphi},S)$ such that $\overline{\varphi}|_Q=\varphi$.
       \item[$\bullet$] A fusion system $\mathcal{F}$ over a $p$-group $S$ is saturated if each subgroup of $S$ is $\mathcal{F}$-conjugate to a subgroup which is fully automized and receptive.
\end{itemize}
\end{defn}
The fusion category $\mathcal{F}_S(G)$ of a finite group $G$ with morphisms given by conjugation in $G$ and objects the subgroups of a Sylow $p$-subgroup $S$ in $G$ is the first example of saturated fusion systems. A different, important example which will be given in detail in Section \ref{sec3} is the saturated fusion system associated to a block algebra $b$ of a finite group $G$, with morphisms given by conjugation between $b$-Brauer pairs.

The theory of fusion systems is an emerging area of mathematics and we copy a paragraph from the Introduction of \cite{AsKeOl} to emphasize this aspect:
"Puig created his theory of Frobenius categories largely as a tool in modular representation theory, motivated in part by work of Alperin and Brou\'{e}. Later, homotopy theorists used this theory to provide a formal setting for, and prove results about, the $p$-completed classifying spaces of finite
groups. As part of this process, objects called \emph{p-local finite groups} associated to abstract fusion systems were introduced by Broto, Levi and Oliver in \cite{BLO}; these also possess interesting $p$-completed classifying spaces. Finally, local finite group theorists became interested in fusion systems, in part because methods from local group theory proved to be effective in the study of fusion systems, but also because certain results in finite group theory seem to be easier to prove in the category of saturated fusion systems."

\section{Proofs of theorems from Section \ref{sec1} }\label{sec2}

\emph{Proof of Theorem \ref{thmbij}.} Clearly $\Ker(i)$ is a nilpotent ideal and $\Hc^*(\mathcal{G})$ is finitely generated as $\Hc^*(\mathcal{F})$-module. This yields the surjectivity of $i^*$.

Suppose that $\mathcal{G}$ controls $p$-fusion in $\mathcal{F}$ on elementary abelian $p$-subgroups. Let $m_1,m_2\in V_{\mathcal{G}}$ such that $i^*(m_1)=i^*(m_2)$. By Quillen stratification for $V_{\mathcal{G}}$ (Theorem \ref{thmA1}), there are two elementary abelian $p$-subgroups $E_1,E_2\leq P$ unique up to $\mathcal{G}$-isomorphism and $\gamma_1\in V_{E_1}^+, \gamma_2\in V_{E_2}^+$ such that
$$m_1=r_{\mathcal{G},E_1}^*(\gamma_1), m_2=r_{\mathcal{G},E_2}^*(\gamma_2)$$
hence $$(i^*\circ r_{\mathcal{G},E_1}^*)(\gamma_1)=(i^*\circ r_{\mathcal{G},E_2}^*)(\gamma_2),$$
and since $r_{\mathcal{G},E_1}\circ i=r_{\mathcal{F},E_1}, r_{\mathcal{G},E_2}\circ i=r_{\mathcal{F},E_2}$ we obtain that $r_{\mathcal{F},E_1}^*(\gamma_1)=r_{\mathcal{F},E_2}^*(\gamma_2)$. Quillen stratification for $V_{\mathcal{F}}$ (Theorem \ref{thmA1}) gives us that $E_1$ is $\mathcal{F}$-isomorphic  to $E_2$ hence $E_1$ is $\mathcal{G}$-isomorphic to $E_2$. This allow us to choose $E_1=E_2=E$ such that $r_{\mathcal{F},E}^*(\gamma_1)=r_{\mathcal{F},E}^*(\gamma_2)$. The inseparable isogeny from Theorem \ref{thmA1}, (ii) is given by $r_{\mathcal{F},E}^*$ so $\gamma_1, \gamma_2$ are in the same orbit of the action of $W_{\mathcal{F}}(E)$ on $V_E^+$. The control of fusion on elementary abelian $p$-subgroups assure us that $\gamma_1,\gamma_2$ are in the same orbit of the action of $W_{\mathcal{G}}(E)$ on $V_E^+$. In conclusion
$$m_1=r_{\mathcal{G},E}^*(\gamma_1)=r_{\mathcal{G},E}^*(\gamma_2)=m_2.$$

Conversely suppose that $i^*$ is injective. First we prove that if $E_1, E_2$ are $\mathcal{F}$-isomorphic elementary abelian $p$-subgroups then they are also $\mathcal{G}$-isomorphic. So let $E_1, E_2$ be $\mathcal{F}$-isomorphic. Then $r_{\mathcal{F},E_1}^*(V_{E_1}^+)=r_{\mathcal{F},E_2}^*(V_{E_2}^+)$, hence $$i^*(r_{\mathcal{G},E_1}^*(V_{E_1}^+))=i^*(r_{\mathcal{G},E_2}^*(V_{E_2}^+)).$$ Since $i^*$ is injective  we obtain that $V_{\mathcal{G},E_1}^+=V_{\mathcal{G},E_2}^+$, hence $E_1,E_2$ are $\mathcal{G}$-isomorphic.

Secondly we prove that $\Aut_{\mathcal{F}}(E)=\Aut_{\mathcal{G}}(E)$ for any elementary abelian $p$-subgroup $E$ of $P$. Since $V_{\mathcal{F},E}^+=i^*(V_{\mathcal{G},E}^+)$ we obtain that $i^*$ induces a bijection between $V_{\mathcal{F},E}^+$ and $V_{\mathcal{G},E}^+$. This bijection, the inclusion $\Aut_{\mathcal{G}}(E)\subseteq\Aut_{\mathcal{F}}(E)$, Remark \ref{remfree} and the similar statements for $V_{\mathcal{G},E}^+$ give the desired equality;  the definitions of the bijection from $V_{\mathcal{F},E}^+$ to $V_E^+/\Aut_{\mathcal{F}}(E)$ (Theorem \ref{thmA1}) and of $i^*$ are important for showing this equality.

Finally, let $\varphi\in\Hom_{\mathcal{F}}(E_1,E_2)$, which gives us the decomposition
$$\xymatrix{E_1 \ar[rr]^{\varphi_1}&&\varphi(E_1)\ar@{^{(}->}[rr] &&E_2},$$ where  $\varphi_1:E_1\rightarrow \varphi(E_1)$ is an isomorphism in $\mathcal{F}$ hence, from the above, we have that $E_1$ is $\mathcal{G}$-isomorphic to $\varphi(E_1)$. It follows that there is $\alpha:E_1\rightarrow\varphi(E_1)$ an isomorphism in $\mathcal{G}$ such that $\alpha^{-1}\circ \varphi_1\in \Aut_{\mathcal{F}}(E_1)$, that is $\alpha^{-1}\circ \varphi_1\in \Aut_{\mathcal{G}}(E_1)$. It is easy to see now that $\varphi\in \Hom_{\mathcal{G}}(E_1,E_2)$.

\emph{Proof of Theorem \ref{thmMislin}.} From the hypothesis we obtain that $i^*$ is an $F$-isomorphism. In particular $i^*$ is bijective. By Theorem \ref{thmbij} we get that $\mathcal{G}$ controls $p$-fusion in $\mathcal{F}$ on elementary abelian $p$-subgroups. Hence \cite[Theorem B]{BenGroHen} assure us that $\mathcal{G}=\mathcal{F}$.

\section{Applications of Mislin's theorem to block algebras of finite groups}\label{sec3}
Let $H,G$ be two finite groups with a common $p$-subgroup $P$. Let $b$ be a block of $kG$  and let $c$ be a block of $kH$ with the same defect group  $P$. Let $P_{\gamma},P_{\delta}$ be defect pointed groups of $G_{\{b\}}$, respectively $H_{\{c\}}$ and $i\in \gamma, j\in \delta$ some source idempotents. Let $(P,e_P)$ be a maximal $(G,b)$-Brauer pair associated to $P_{\gamma}$ and $(P,f_P)$ be a maximal $(H,c)$-Brauer pair associated to $P_{\delta}$. For any subgroup $R$ of $P$, there is a unique block $e_R$ of $C_G(R)$ such that $\Br_R(i)e_R \neq 0$. Then $(R, e_R)$ is a $(G,b)$-Brauer pair and $e_R$ is also the unique block of $C_G(R)$ such that $(R, e_R) \leq (P, e_P)$. We define $\mathcal{F}_{(P,e_P)}(G,b)$ as the  category which has as objects the set of subgroups of $P$; for any two subgroups $R,S$ of $P$ the set of morphisms from $R$ to $S$ in  $\mathcal{F}_{(P,e_P)}(G,b)$ is the set of (necessarily injective) group homomorphisms $\varphi:R\rightarrow S$ for which there is an element $x \in G$ satisfying $\varphi(u)=xux^{-1}$ for all $u\in R$ and satisfying $^x(R, e_R)\leq (S, e_S)$. The category $\mathcal{F}_{(P,e_P)}(G,b)$ is sometimes called the Brauer category of $b$ with respect to the choice of $(P,e_P)$ and is a saturated fusion system. The analogous definitions give $\mathcal{F}_{(P,f_P)}(H,c)$. For the rest of this section we assume that $\mathcal{F}_{(P,f_P)}(H,c)$ is a subsystem of $\mathcal{F}_{(P,e_P)}(G,b)$ and $p$ is an odd prime.

We refer the reader to \cite{LiTr} for results regarding transfer maps and stable elements in Hochschild cohomology algebras. Recall that if $A,B$ are two symmetric $k$-algebras and $X$ is a bounded complex of $A-B$-bimodules whose components are projective as left and right modules, there is a graded $k$-linear map $t_X:\HH^*(B)\rightarrow\HH^*(A)$ called the transfer map associated to $X$ \cite[Definition 2.9]{LiTr}. If $\pi_X=t_X^0(1_B)\in Z(A)$ (the relatively $X$-projective element) is invertible then $T_X=\pi_X^{-1}t_X$ is called the normalized transfer map associated to $X$. If both $\pi_X,\pi_{X^*}$ are invertible then $T_X$ induces a graded $k$-algebra isomorphism which by abuse of notation we denote still by $T_X:\HH^*_{X^*}(B)\rightarrow\HH^*_X(A)$, with its inverse $T_{X^*}$. Here $\HH^*_X(A)$ is the graded subalgebra of $X$-stable elements in $\HH^*(A)$; more precisely  $\zeta\in \HH^*_{X}(A)$ if there is $\theta\in\HH^*(B)$ such that $\zeta\otimes\Id_{X}=\Id_{X}\otimes \theta$ in $\Ext_{A\otimes B^{op}}(X,X)$, see \cite[Definition 2]{Sas}. For example, we denote by $\HH^*_{ikGi}(kP)$ (respectively $\HH^*_{jkHj}(kP)$) the subalgebra of $ikGi$-stable ($jkHj$-stable) elements in the Hochschild cohomology algebra of $kP$. The $k$-algebra $ikGi$ (respectively $jkHj$) is called the source algebra of $b$ (respectively $c$) and we obtain decompositions as $kP-kP$-bimodules
$$ikGi\cong\bigoplus_{g\in Y_{G,b}}k[PgP],~~~jkHj\cong\bigoplus_{h\in Y_{H,c}}k[PhP]$$
where $Y_{G,b}\subseteq [P\setminus G/ P]$ and $Y_{H,c}\subseteq [P\setminus H/P]$, see \cite[Theorem 44.3]{TH}.

Next proposition may be regarded as a tool for checking when two blocks of two finite groups have the same local structure and  is a first application of Mislin's theorem (Theorem \ref{thmMislin}) to block algebras. Recall that $\delta_P:\Hc^*(P,k)\rightarrow\HH^*(kP)$ is the embedding defined in \cite[Proposition 4.5]{LiTr}.
\begin{prop}\label{propMorita-Brauer} With the above assumptions and notations if $$\HH^*_{jkHj}(kP)\cap \delta_P(\Hc^*(P,k))\subseteq \HH^*_{ikGi}(kP)$$ then $\mathcal{F}_{(P,f_P)}(H,c)=\mathcal{F}_{(P,e_P)}(G,b)$. In particular if $ikGi$ is isomorphic with a direct summand of $jkHj$ as $kP-kP$-bimodules then $\mathcal{F}_{(P,f_P)}(H,c)=\mathcal{F}_{(P,e_P)}(G,b)$.
\end{prop}
\begin{proof} The second statement follows from the first if we use \cite[Proposition 3.5, (iv)]{LiTr}. For the first statement, we know from \cite[Proposition 5.4]{LiTr} that $$\delta_P(\Hc^*(\mathcal{F}_{(P,e_P)}(G,b)))\subseteq \HH^*_{ikGi}(kP).$$ Moreover \cite[Theorem 1]{Sas} is a nice result which tells us more precisely that
$$\delta_P(\Hc^*(\mathcal{F}_{(P,e_P)}(G,b)))=\HH^*_{ikGi}(kP)\cap\delta_P(\Hc^*(P,k)),$$
$$\delta_P(\Hc^*(\mathcal{F}_{(P,f_P)}(H,c)))= \HH^*_{jkHj}(kP)\cap\delta_P(\Hc^*(P,k)).$$
It follows that $\delta_P(\Hc^*(\mathcal{F}_{(P,e_P)}(G,b)))=\delta_P(\mathcal{F}_{(P,f_P)}(H,c))$, hence $$\Hc^*(\mathcal{F}_{(P,e_P)}(G,b))=\Hc^*(\mathcal{F}_{(P,f_P)}(H,c)).$$ Now, Theorem \ref{thmMislin} gives us our desired conclusion.
\end{proof}
\begin{rem}  If $kGb$ and $kHc$ are two block algebras  for which there is an indecomposable direct summand $M$ of the $kGb-kHc$-bimodule $kGi\otimes_{kP}jkH$ which induces a Morita equivalence (i.e. there is a Morita equivalence between the blocks induced by a $p$-permutation bimodule); see \cite[Theorem 4.1]{LiSpl} or \cite[7.5.1]{Pu}, then we actually have an isomorphism of $kP-kP$-bimodules between the source algebras. It follows that the above proposition gives an alternative proof (when p is odd) of
the fact that a splendid Morita equivalence preserves block fusion systems.
\end{rem}

It is well known that if $kHc$ and $kGb$ are splendid stable equivalent then their cohomology algebras are isomorphic $\Hc^*(\mathcal{F}_{(P,e_P)}(G,b))\cong \Hc^*(\mathcal{F}_{(P,f_P)}(H,c)).~~~$ With our assumptions Mislin's theorem applied to block algebras gives a new method to prove that if two block algebras are splendid stable equivalent then the associated fusion systems are the same.
\begin{thm}\label{thmsplendidstable}Let $b,c$ be two blocks with the above assumptions. Let $X$ be a bounded complex of $kHc-kGb$-bimodules whose components are isomorphic to direct sums of direct summands of $kHj\otimes_{kQ}ikG$, where $Q$ runs over the set of subgroups of $P$. If $X$ induces a splendid stable equivalence (i.e. $X\otimes_{kGb}X^*\cong kHc\oplus U_c, X^*\otimes_{kHc}X\cong kGb\oplus U_b$, where $U_c$ is a bounded complex of projective $kHc-kHc$-bimodules and $U_b$ is a bounded complex of projective  $kGb-kGb$-bimodules) then $\mathcal{F}_{(P,f_P)}(H,c)=\mathcal{F}_{(P,e_P)}(G,b)$.
\end{thm}
\begin{proof} We mimic the proof of \cite[Theorem 5.5., (ii)]{LiVar} to obtain the following commutative diagram
\begin{displaymath}
 \xymatrix{\Hc^*(\mathcal{F}_{(P,e_P)}(G,b))\ar[rr]^{ T_{kGi}\circ\delta_P}\ar@{^{(}->}[d]^{i} && \HH^*_{X^*}(kGb) \ar[d]^{T_X} \\
                                                     \Hc^*(\mathcal{F}_{(P,f_P)}(H,c))\ar[rr]^{T_{kHj}\circ\delta_P}
                                                     &&\HH^*_{X}(kHc)
                                            }.
\end{displaymath}
where $T_X,T_{kGi}, T_{kHj}$  are the normalized transfer maps. Since the projective elements $\pi_X,\pi_{X^*}$ are invertible (\cite[Theorem 5.5, (i)]{LiVar}) the map $T_X$ is an isomorphism. For completeness we repeat some of the arguments of that proof. By \cite[3.2.3]{LiTr} we may choose symmetrizing forms such that $\pi_{kGi}=1_{kGb}, \pi_X=1_{kHc}$, or equivalently $T_{kGi}=t_{kGi}, T_X=t_X$. The relatively projective elements $\pi_{kHj}, \pi_{jkH}$ are still invertible, \cite[Proposition 5.4, (iv)]{LiTr}. In order to show the commutativity of the above diagram we need to show that $T_{jkH}\circ T_X\circ T_{kGi}\circ \delta_P=\delta_P\circ i$. This is equivalent to showing that $$\pi_{jkH}^{-1}\circ t_{jkH}\circ t_X\circ t_{kGi}\circ \delta_P=\delta_P\circ i,$$ hence equivalent to $$\pi_{jkH}^{-1}\circ t_{jXi}\circ \delta_P=\delta_P\circ i .$$ With our notations by \cite[Proposition 5.7, (iv)]{LiTr} we know that the map $t_{jXi}$ acts as multiplication by $\pi_{jXi}$ on $\delta_P(\Hc^*(\mathcal{F}_{(P,e_P)}(G,b)))$, so all we need to show is that $\pi_{jXi}=\pi_{jkH}$. But this is true since by \cite[3.2]{LiTr} we have $$\pi_{jXi}=t_{jXi}^0(1_{kP})=(t_{jkH}^0\circ t_{X}^0\circ t_{kGi}^0)(1_{kP})=t_{jkH}^0(t_{X}^0(1_{kGb}))=t_{jkH}^0(1_{kHc})=\pi_{jkH}.$$

We denote by $\tau_b$ the injective graded $k$-algebra homomorphism $T_{kGi}\circ \delta_P$. Similarly $\tau_c$ is $T_{kHj}\circ \delta_P$. By \cite[Theorem 1.1]{LiVariso} $\tau_b,\tau_c$ induce  isomorphisms of varieties
$$\tau_b^*:X_{kGb}\rightarrow V_{\mathcal{F}_{(P,e_P)}(G,b)}$$
$$\tau_c^*:X_{kHc}\rightarrow V_{\mathcal{F}_{(P,f_P)}(H,c)}$$
hence $\tau_b^*,\tau_c^*$ are bijective maps; where $X_{kGb},X_{kHc}$ are the  varieties of the Hochschild cohomology algebras $\HH^*(kGb), \HH^*(kHc)$. The first diagram yields $$T_X\circ \tau_b=\tau_c\circ i$$ hence $$\tau_b^*\circ T_X^*=i^*\circ\tau_c^*.$$ Since $T_X$ is an isomorphism it follows that $i^*$ is injective. Now Theorem \ref{thmMislin} give us the conclusion.
\end{proof}
In the last remark of this paper we give a new consequence of Mislin's theorem. We obtain a generalization of a theorem of Watanabe \cite[Theorem 2]{Wat} to the class of all finite groups, not just for $p$-solvable groups.
\begin{rem}
Let $R$ be a normal subgroup of $P$ such that $N_G(P)\leq N_G(R)$. We denote by $d$ the Brauer correspondent of $b$ in $N_G(R)$, that is $d$ is the unique block of $N_G(R)$ such that $\Br_P(d)=\Br_P(b)$, where $\Br_P$ is the Brauer homomorphism from $(kG)^P$ to $kC_G(P)$. Set $N=N_G(R)$. Then $b=d^G$ and $(P,e_P)$ is a maximal $(N,d)$-Brauer pair. Moreover $\mathcal{F}_{(P,e_P)}(N,d)$ is a subsystem of $\mathcal{F}_{(P,e_P)}(G,b)$. Now  from Theorem \ref{thmMislin}  we have  that $\Hc^*(\mathcal{F}_{(P,e_P)}(N,d))=\Hc^*(\mathcal{F}_{(P,e_P)}(G,b))$ if and only if $\mathcal{F}_{(P,e_P)}(N,d)=\mathcal{F}_{(P,e_P)}(G,b)$.
\end{rem}

\textbf{Acknowledgments.} We are grateful to Professor Markus Linckelmann for some email discussions and for the reference \cite{LiQuillfus}.

\end{document}